\documentclass[11pt]{article}
\usepackage{amsmath, amssymb, amsfonts, amsthm, latexsym} 
\usepackage{mathrsfs}
\usepackage{graphicx}
\usepackage{authblk}
\usepackage[usenames]{color}
\newtheorem{theorem}{Theorem}[section]
\newtheorem{lemma}[theorem]{Lemma}
\newtheorem{claim}[theorem]{Claim}
\newtheorem{conjecture}[theorem]{Conjecture}
\newtheorem{definition}[theorem]{Definition}
\newtheorem{proposition}[theorem]{Proposition}

\newtheorem{remark}[theorem]{Remark}
\newtheorem{observation}[theorem]{Observation}

\newtheorem*{notat*}{Notation}

\title{On graphs whose cycle space is spanned by their Hamilton cycles}

\author{Dan Hefetz \thanks{School of Computer Science, Ariel University, Ariel 40700, Israel. Email: {\tt danhe@ariel.ac.il}.}
\quad Michael Krivelevich \thanks{School of Mathematical Sciences, Tel Aviv University, Tel Aviv 6997801, Israel. Research supported in part
by NSF-BSF grant 2023688. Email: {\tt krivelev@tauex.tau.ac.il}.}}

\begin{document}
\maketitle
 
\begin{abstract}
The cycle space of a graph $G$, denoted $\mathcal{C}(G)$, is a vector space over ${\mathbb F}_2$, spanned by all incidence vectors of edge-sets of cycles of $G$. If $G$ has $n$ vertices, then $\mathcal{C}_n(G)$ denotes the subspace of $\mathcal{C}(G)$, spanned by the incidence vectors of Hamilton cycles of $G$. We consider several known sufficient conditions for Hamiltonicity and show that an appropriate and fairly mild strengthening of each such condition in fact ensures the stronger property $\mathcal{C}_n(G) = \mathcal{C}(G)$. In particular, we consider the classical Chv\'atal-Erd\H{o}s criterion and prove that (under various additional restrictions) if $n$ is odd and $\kappa(G) \geq c \alpha(G)$, where $c$ is a sufficiently large absolute constant, then $\mathcal{C}_n(G) = \mathcal{C}(G)$. Moreover, considering the McDiarmid-Yolov criterion we prove that if $n$ is odd and $\delta(G) \geq \max \left\{2 \tilde{\alpha}(G) + 9, \tilde{\alpha}(G) + 18 \right\}$, where $\tilde{\alpha}(G)$ is the so-called bipartite independence number of $G$, then $\mathcal{C}_n(G) = \mathcal{C}(G)$. We also prove that if $n$ is odd and $G$ admits $16 \alpha(G) + 12$ pairwise disjoint connected dominating sets, then $\mathcal{C}_n(G) = \mathcal{C}(G)$. Finally, we consider an effective Chv\'atal-Erd\H{o}s type criterion for bipartite graphs and prove that if $G$ is a balanced bipartite graph on $2n$ vertices, satisfying $\alpha_{\emph{BIP}}(G) \leq 2 \delta(G) - 24$, then $\mathcal{C}_{2n}(G) = \mathcal{C}(G)$.           
\end{abstract}

\section{Introduction}  
Let $G = (V,E)$ be a graph on $n$ vertices. The \emph{edge space} of $G$, denoted $\mathcal{E}(G)$, is a vector space over ${\mathbb F}_2$ consisting of all incidence vectors of subsets of $E$. The \emph{cycle space} of $G$, denoted $\mathcal{C}(G)$, is the subspace of $\mathcal{E}(G)$, spanned by all incidence vectors of cycles of $G$. For any integer $3 \leq k \leq n$, let $\mathcal{C}_k(G)$ be the subspace of $\mathcal{C}(G)$, spanned by all incidence vectors of cycles of length $k$ in $G$. Determining conditions under which $\mathcal{C}_k(G) = \mathcal{C}(G)$ holds for some $3 \leq k \leq n$ is a well-studied problem (see, e.g.,~\cite{BK, BL, DHJ, Hartman, Locke1, Locke2}). In this paper we are interested in the case $k = n$, that is, in graphs whose cycle space is spanned by their Hamilton cycles. This problem has been addressed by various researchers (see, e.g.,~\cite{ALW, HK, HKGnd, Heinig1, Heinig2, HY}). Since the symmetric difference of any two even graphs (i.e., subsets of $E$ of even size) is an even graph, it is evident that if $\mathcal{C}_n(G) = \mathcal{C}(G)$, then $G$ is bipartite or $n$ is odd. Moreover, $G$ must either be acyclic or Hamiltonian.\footnote{In fact, if $G$ is Hamiltonian, then it is  also 2-connected, implying that every edge of $G$ is in some cycle of $G$. Hence, if $G$ is Hamiltonian and $\mathcal{C}_n(G) = \mathcal{C}(G)$, then any edge of $G$ is in a Hamilton cycle of $G$. The latter property lies between being Hamiltonian and being Hamilton-connected.} Since the former case is trivial, this study can be viewed as part of the common theme of proving that (possibly somewhat strengthened) various sufficient conditions for Hamiltonicity in fact ensure stronger properties. 

One of the earliest and most well-known sufficient conditions for Hamiltonicity is Dirac's Theorem, stipulating that any graph on $n \geq 3$ vertices whose minimum degree is at least $n/2$ is Hamiltonian; $n$-vertex graphs whose minimum degree is at least $n/2$ are thus dubbed \emph{Dirac graphs}. Given an $n$-vertex graph $G$, where $n$ is odd, with minimum degree $(n+1)/2 + k$, where $k \geq 0$ is as small as possible, we would like to determine whether $\mathcal{C}_n(G) = \mathcal{C}(G)$ holds. This question was first considered by Heinig~\cite{Heinig1} who proved that $\mathcal{C}_n(G) = \mathcal{C}(G)$ holds whenever $G$ is an $n$-vertex graph with minimum degree $\delta(G) \geq (1/2 + \varepsilon) n$, for any constant $\varepsilon > 0$ and sufficiently large odd $n$. This was significantly improved by Christoph, Nenadov, and Petrova~\cite{CNP} who proved that $\delta(G) \geq n/2 + 37$ suffices. Recently, a complete resolution of this problem was announced by Hou and Yin~\cite{HY} who proved that $\delta(G) > n/2$ is sufficient (and necessary) for $\mathcal{C}_n(G) = \mathcal{C}(G)$ to hold. 

Another natural venue for this problem are random and pseudo-random graphs. Indeed, it was proved by Christoph, Nenadov, and Petrova in~\cite{CNP} that $\mathcal{C}_n(G) = \mathcal{C}(G)$ holds whenever $G$ is a pseudo-random graph with certain appropriate properties (see Theorem 1.3 in~\cite{CNP} for details). One immediate consequence of this result, significantly improving a result of Heinig~\cite{Heinig2}, is that if $G \sim \mathbb{G}(n,p)$, where $n$ is odd and $p \geq C \log n/n$ for a sufficiently large constant $C$, then asymptotically almost surely (a.a.s. for brevity hereafter) $\mathcal{C}_n(G) = \mathcal{C}(G)$. The exact bound on $p$ has been subsequently attained by the authors in~\cite{HK} where it is proved that if $p$ is large enough to ensure that $\delta(G) \geq 3$ holds a.a.s., then it also ensures that a.a.s. $\mathcal{C}_n(G) = \mathcal{C}(G)$ (the necessity of this condition was observed by Heinig in~\cite{Heinig2}). Note that the required value of $p$ is slightly larger than the one required for Hamiltonicity (recall that the latter coincides with the probability ensuring that $\delta(G) \geq 2$ holds a.a.s.).   

An additional immediate consequence of the aforementioned result of Christoph, Nenadov, and Petrova regarding pseudo-random graphs is that if $G$ is an $(n, d, \lambda)$-graph, $n$ is odd, $d \geq C \log n$, and $\lambda \leq \varepsilon d$, where $C$ and $\varepsilon$ are appropriate constants, then $\mathcal{C}_n(G) = \mathcal{C}(G)$. This was somewhat improved by the authors~\cite{HKGnd} who proved that $d \geq C \log n/ \log (d/\lambda)$ suffices to ensure the same conclusion (with the other assumptions unchanged). In the most natural special case, where $G$ is a random $d$-regular graph $G \sim \mathcal{G}_{n, d}$, it is proved in~\cite{HKGnd} that sufficiently large yet constant $d$ is a.a.s. enough to ensure that $\mathcal{C}_n(G) = \mathcal{C}(G)$ holds. Since odd $n$ mandates that $d$ is even, a complementary result, taking care, in particular, of odd values of $d$, is also proved in~\cite{HKGnd}. Namely, it is proved there that if $G \sim \mathcal{G}_{n, d}$, where $n$ is even and $d$ is a sufficiently large constant, then a.a.s. $\mathcal{C}_{n-1}(G) = \mathcal{C}(G)$.    

In light of the aforementioned results for Dirac graphs and for $\mathbb{G}(n,p)$, it is natural to consider this cycle space problem for randomly perturbed dense graphs. This was done by the authors in~\cite{HKGnd}, where it is proved that if $H$ is an $n$-vertex graph with minimum degree $\delta(H) \geq \delta n$, where $\delta > 0$ is a constant and $n$ is odd, and $G \sim \mathbb{G}(n,p)$, where $p \geq C/n$ for a sufficiently large constant $C := C(\delta)$, then a.a.s. $\mathcal{C}_n(H \cup G) = \mathcal{C}(H \cup G)$ holds.  

\subsection{Our results}

All results stated in this section show that mild strengthenings of known sufficient conditions for an $n$-vertex graph $G$ to be Hamiltonian in fact imply that $\mathcal{C}_n(G) = \mathcal{C}(G)$; throughout we assume (explicitly or implicitly) that $n$ is odd or $G$ is bipartite.

A well-known sufficient condition for Hamiltonicity, due to Chv\'atal and Erd\H{o}s~\cite{CE}, stipulates that highly connected graphs with a small independence number are Hamiltonian. Formally, their result reads as follows.
\begin{theorem} [\cite{CE}] \label{th::HamCE}
Let $G$ be a graph on at least 3 vertices. If $\kappa(G) \geq \alpha(G)$, then $G$ is Hamiltonian. 
\end{theorem}

It was conjectured by Jackson and Ordaz~\cite{JO} that the (needed) slight strengthening $\kappa(G) > \alpha(G)$ in fact ensures that $G$ is pancyclic. Following a recent breakthrough by Dragani\'c, Munh\'a Correia, and Sudakov~\cite{DMS2}, who proved an asymptotically optimal version of this conjecture, it has been completely resolved (for large graphs) by Letzter~\cite{Letzter}. It is natural to ask whether the same (or, possibly, slightly strengthened --- in particular, $n$ must be odd) condition also implies $\mathcal{C}_n(G) = \mathcal{C}(G)$. We prove that this is indeed the case up to a multiplicative constant factor, provided $\kappa(G)$ is not too small as a function of $|V(G)|$.

\begin{theorem} \label{th::CycleSpaceCE}
There exist a constant $c$ and an integer $n_0$ such that the following holds. Let $G$ be an $n$-vertex graph, where $n \geq n_0$ is odd. If $\kappa(G) \geq c \max\{\alpha(G), \log n\}$ or $\kappa(G) \geq c \alpha(G)^2$, then $\mathcal{C}_n(G) = \mathcal{C}(G)$.
\end{theorem}

The condition $\kappa(G) \geq c \log n$ appearing in the statement of Theorem~\ref{th::CycleSpaceCE} is an artifact of our proof of that theorem and we do not believe it is necessary; this leads us to the following conjecture.
\begin{conjecture} \label{conj::CycleSpaceCE}
There exists a constant $c$ such that the following holds. Let $G$ be an $n$-vertex graph, where $n$ is odd. If $\kappa(G) \geq c \alpha(G)$, then $\mathcal{C}_n(G) = \mathcal{C}(G)$.
\end{conjecture}     
The following result offers a first step towards the resolution of Conjecture~\ref{conj::CycleSpaceCE}.


\begin{theorem} \label{th::CycleSpaceCELinearDegree}
Let $G$ be an $n$-vertex graph, where $n$ is a sufficiently large odd integer. If $\delta(G) \geq \delta n$ for some $\delta := \delta(n) > 0$, and $\kappa(G) \geq \max \left\{180 \alpha(G), 2000 \delta^{-2} \right\}$, then $\mathcal{C}_n(G) = \mathcal{C}(G)$.
\end{theorem}

Note that the constants 180 and 2000 appearing in the statement of Theorem~\ref{th::CycleSpaceCELinearDegree} are not optimal. They were chosen in order to make our calculations easy to follow.

It is not hard to see that $c = 1$ is not sufficient in Conjecture~\ref{conj::CycleSpaceCE}. That is, similarly to the case of pancyclicity, $\kappa(G) > \alpha(G)$ is necessary (though perhaps not sufficient) to ensure $\mathcal{C}_n(G) = \mathcal{C}(G)$.
\begin{proposition} \label{prop::CElowerbound}
For every integer $k \geq 2$ and sufficiently large odd $n$ there exists an $n$-vertex graph $G$ satisfying $\kappa(G) = \alpha(G) = k$ and $\mathcal{C}_n(G) \neq \mathcal{C}(G)$.
\end{proposition}

A \emph{connected dominating set} (CDS for brevity, hereafter) in a graph $G$ is a set $S \subseteq V(G)$ such that $G[S]$ is connected and for every vertex $v \in V(G) \setminus S$ there exists a vertex $u \in S$ such that $uv \in E(G)$. CDSs form an important concept in graph theory with many practical applications (see, e.g., the survey~\cite{YWWY} and the many references therein). It is easy to see that if $G$ admits $t$ pairwise disjoint CDSs, then it is $t$-connected. Therefore, essentially the same arguments we use to prove Theorem~\ref{th::CycleSpaceCE} can also be used to prove the following result.

\begin{theorem} \label{th::CycleSpaceCDS}
Let $G$ be an $n$-vertex graph, where $n$ is odd, which admits $16 \alpha(G) + 12$ pairwise disjoint CDSs. Then, $\mathcal{C}_n(G) = \mathcal{C}(G)$.
\end{theorem}

Note that the same example we present in the proof of Proposition~\ref{prop::CElowerbound} shows that Theorem~\ref{th::CycleSpaceCDS} is tight up to the constants $16$ and $12$. 


\begin{proposition} \label{prop::CDSlowerbound}
For every integer $k \geq 2$ and sufficiently large odd $n$, there exists an $n$-vertex graph $G$ which admits $\alpha(G)$ pairwise disjoint CDSs but satisfies $\mathcal{C}_n(G) \neq \mathcal{C}(G)$.
\end{proposition}

Another interesting sufficient condition for Hamiltonicity, which extends Dirac's classical criterion, was proved by McDiarmid and Yolov~\cite{MY}. In order to state their result, we need the following definition.

\begin{definition} \label{def::BipartiteIndependence}
Let $G$ be a graph. The \emph{bipartite independence number} of $G$, denoted $\tilde{\alpha}(G)$, is the smallest integer $k$ for which there exist positive integers $a$ and $b$ such that $a + b = k + 1$, and $E_G(A,B) \neq \emptyset$ whenever $A \subseteq V(G)$ and $B \subseteq V(G) \setminus A$ are sets of sizes $|A| = a$ and $|B| = b$.
\end{definition}

We may now state the aforementioned result of McDiarmid and Yolov.
\begin{theorem} [\cite{MY}] \label{th::HamMY}
Let $G$ be a graph on at least 3 vertices. If $\delta(G) \geq \tilde{\alpha}(G)$, then $G$ is Hamiltonian. 
\end{theorem}

Note that, as alluded to above, this is a generalization of Dirac's Theorem. Indeed, If $\delta(G) \geq \lceil n/2 \rceil$, then $E_G(v, A) \neq \emptyset$ for every $v \in V(G)$ and every $A \subset V(G) \setminus \{v\}$ of size $\lfloor n/2 \rfloor$. Hence, $\tilde{\alpha}(G) \leq \lceil n/2 \rceil \leq \delta(G)$ and thus $G$ is Hamiltonian by Theorem~\ref{th::HamMY}. This theorem of McDiarmid and Yolov has been significantly strengthened by Dragani\'c, Munh\'a Correia, and Sudakov~\cite{DMS} who proved that the same condition in fact implies that $G$ is pancyclic (with $G = K_{n/2, n/2}$ being the sole exception). Here too we wish to determine whether this condition also implies $\mathcal{C}_n(G) = \mathcal{C}(G)$. We prove that this is indeed the case up to a small multiplicative constant factor.

\begin{theorem} \label{th::CycleSpaceBI}
Let $G$ be an $n$-vertex graph, where $n$ is odd. If $\delta(G) \geq \max \left\{2 \tilde{\alpha}(G) + 9, \tilde{\alpha}(G) + 18 \right\}$, then $\mathcal{C}_n(G) = \mathcal{C}(G)$.
\end{theorem}

Note that both the Chv\'atal-Erd\H{o}s and the McDiarmid-Yolov sufficient conditions for Hamiltonicity (namely, Theorems~\ref{th::HamCE} and~\ref{th::HamMY}) are trivial for bipartite graphs. Indeed, if $G$ is a bipartite graph satisfying $\kappa(G) \geq \alpha(G)$ or $\delta(G) \geq \tilde{\alpha}(G)$, then $G$ is a balanced complete bipartite graph and thus obviously Hamiltonian. In an attempt to devise an effective variant of the Chv\'atal-Erd\H{o}s criterion for bipartite graphs, the following parameter was defined by Ash~\cite{Ash}.  

\begin{definition} \label{def::AlphaBIP}
Let $G = (X \cup Y, E)$ be a balanced bipartite graph, i.e. $|X| = |Y|$. An independent set $A \subseteq X \cup Y$ is said to be \emph{balanced} if $|A \cap X| = |A \cap Y|$. The balanced independence number of $G$, denoted $\alpha_{\emph{BIP}}(G)$, is the largest size of a balanced independent set in $G$. 
\end{definition}

Following results by Ash~\cite{Ash}, by Fraisse~\cite{Fraisse}, and by Favaron, Mago, and Ordaz~\cite{FMO}, the following optimal result was proved by Ordaz, Amar, and Raspaud. 



\begin{theorem} [\cite{OAR}] \label{th::HamBipOAR}
Let $G$ be a balanced bipartite graph. If $\alpha_{\emph{BIP}}(G) \leq 2 \delta(G) - 4$, then $G$ is Hamilton-biconnected.\footnote{See the definition of a Hamilton-biconnected graph in Section~\ref{sec::prelim}.}
\end{theorem}

As alluded to above, this result is tight. Indeed, it is shown in~\cite{OAR} that there are infinitely many values of $n$ for which there exist non-Hamiltoian balanced bipartite graphs $G$ on $2n$ vertices which satisfy $\alpha_{\textrm{BIP}}(G) = 2 \delta(G) - 2$. 

Since $\alpha_{\textrm{BIP}}(G) \leq \alpha(G)$ and $\delta(G) \geq \kappa(G)$ clearly hold, Theorem~\ref{th::HamBipOAR} may be viewed as an effective Chv\'atal-Erd\H{o}s criterion for balanced bipartite graphs. 



Our next result asserts that (a minor strengthening of) the condition appearing in Theorem~\ref{th::HamBipOAR} in fact ensures that the cycle space of $G$ is spanned by its Hamilton cycles.
\begin{theorem} \label{th::BipCEcycleSpace} 
Let $G$ be a balanced bipartite graph on $2n$ vertices. If $\alpha_{\emph{BIP}}(G) \leq 2 \delta(G) - 24$, then $\mathcal{C}_{2n}(G) = \mathcal{C}(G)$.
\end{theorem}

Instead of starting with a sufficient condition for Hamiltonicity and show that it can be slightly strengthened so as to yield stronger properties, one may simply start with a Hamiltonian graph. This was done, for example, in the context of pancyclicity. Indeed, a well-known problem of Erd\H{o}s~\cite{Erdos} asks for the smallest function $f$ such that any Hamiltonian $n$-vertex graph $G$ satisfying $n \geq f(\alpha(G))$ is pancyclic (note that, on its own, the condition $n \geq f(\alpha(G))$ does not even imply connectivity, so assuming that $G$ is Hamiltonian is needed). Following a series of results, this problem was recently solved by Dragani\'c, Munh\'a Correia, and Sudakov~\cite{DMS3} who proved that $n \geq (2 + o(1)) \alpha(G)^2$ is sufficient (a construction of Erd\H{o}s shows that this is asymptotically best possible). It is natural to ask whether the same condition (possibly somewhat strengthened) also implies $\mathcal{C}_n(G) = \mathcal{C}(G)$ (assuming of course that $n$ is odd). However, the same example we present in the proof of Propositions~\ref{prop::CElowerbound} and~\ref{prop::CDSlowerbound} shows that this is false in a strong sense.
\begin{proposition} \label{prop::Erdoslowerbound}
For every integer $k \geq 2$, function $f$, and sufficiently large odd $n$, there exists an $n$-vertex Hamitonian graph $G$ satisfying $\alpha(G) = k$, $n \geq f(k)$, and $\mathcal{C}_n(G) \neq \mathcal{C}(G)$.
\end{proposition}

The rest of this paper is organised as follows. In Section~\ref{sec::prelim} we introduce some terminology, notation, and standard tools, and present the method of Christoph, Nenadov, and Petrova from~\cite{CNP} which is a central ingredient in our proofs. In Section~\ref{sec::CEcycleSpace} we prove Theorems~\ref{th::CycleSpaceCE}, \ref{th::CycleSpaceCELinearDegree}, and~\ref{th::CycleSpaceCDS}, and present a construction that proves Propositions~\ref{prop::CElowerbound}, \ref{prop::CDSlowerbound}, and~\ref{prop::Erdoslowerbound}. In Section~\ref{sec::BIcycleSpace} we prove Theorem~\ref{th::CycleSpaceBI} and in Section~\ref{sec::bipartite} we prove Theorem~\ref{th::BipCEcycleSpace}. 


\section{Preliminaries and tools} \label{sec::prelim}

For the sake of simplicity and clarity of presentation, we do not make a particular effort to optimize some of the constants obtained in our proofs. We also omit floor and ceiling signs whenever these are not crucial. Throughout this paper, $\log$ stands for the natural logarithm, unless explicitly stated otherwise. Our graph-theoretic notation is standard; in particular, we use the following.

For a graph $G$, let $V(G)$ and $E(G)$ denote its sets of vertices and edges respectively, and let $v(G) = |V(G)|$ and $e(G) = |E(G)|$. For a set $A \subseteq V(G)$, let $E_G(A)$ denote the set of edges of $G$ with both endpoints in $A$ and let $e_G(A) = |E_G(A)|$. For disjoint sets $A, B \subseteq  V(G)$, let $E_G(A, B)$ denote the set of edges of $G$ with one endpoint in $A$ and one endpoint in $B$, and let $e_G(A, B) = |E_G(A, B)|$. For a set $S \subseteq V(G)$, let $G[S]$ denote the subgraph of $G$ induced by the set $S$, and let $G \setminus S = G[V(G) \setminus S]$. For a set $S \subseteq V(G)$, let $N_G(S) = \{v \in V(G) \setminus S : \exists u \in S \textrm{ such that } uv \in E(G)\}$ denote the \emph{external neighbourhood} of $S$ in $G$. For a vertex $u \in V(G)$ we abbreviate $N_G(\{u\})$ under $N_G(u)$ and let $\textrm{deg}_G(u) = |N_G(u)|$ denote the degree of $u$ in $G$. The maximum degree of a graph $G$ is $\Delta(G) := \max\{\textrm{deg}_G(u) : u \in V(G)\}$, and the minimum degree of a graph $G$ is $\delta(G) := \min \{\textrm{deg}_G(u) : u \in V(G)\}$. For a vertex $u \in V(G)$ and a set $S \subseteq V(G)$, let $\textrm{deg}_G(u, S) = |N_G(u) \cap S|$. Given any two (not necessarily distinct) vertices $x, y \in V(G)$, the \emph{distance} between $x$ and $y$ in $G$, denoted $\textrm{dist}_G(x,y)$, is the length of a shortest path between $x$ and $y$ in $G$, where the length of a path is the number of its edges (for the sake of formality, we define $\textrm{dist}_G(x,y)$ to be $\infty$ whenever $x$ and $y$ lie in different connected components of $G$). The \emph{diameter} of $G$, denoted $\textrm{diam}(G)$, is $\max \{\textrm{dist}_G(x,y) : x, y \in V(G)\}$. A graph $G$ is said to be \emph{Hamilton-connected} if for every two vertices $x, y \in V(G)$ there is a Hamilton path of $G$ whose endpoints are $x$ and $y$. A bipartite graph $G = (X \cup Y, E)$ is said to be \emph{Hamilton-biconnected} if for every two vertices $x \in X$ and $y \in Y$ there is a Hamilton path of $G$ whose endpoints are $x$ and $y$. 


\subsection{The recipe of Christoph, Nenadov, and Petrova} \label{sec::recipe}

Let $n$ be an odd integer, and let $G$ be an $n$-vertex Hamiltonian graph. A recipe for proving $\mathcal{C}_n(G) = \mathcal{C}(G)$ is presented in~\cite{CNP}; it has since been utilized in~\cite{HK, HKGnd, HY}. In order to describe it we need some definitions and results. 

\begin{lemma} [\cite{CNP}] \label{lem::subgraphR}
Let $G$ be an $n$-vertex Hamiltonian graph, where $n$ is odd or $G$ is bipartite, and suppose that $\mathcal{C}_n(G) \neq \mathcal{C}(G)$. Then, there
exists a subgraph $R$ of $G$ such that the following conditions hold.
\begin{enumerate}
\item [\emph{(C1)}] $R \neq G$;

\item [\emph{(C2)}] Every Hamilton cycle in $G$ contains an even number of edges from $R$;

\item [\emph{(C3)}] For every partition $V(G) = A \cup B$ it holds that $e_R(A, B) \geq e_G(A, B)/2$ and $R \neq G[A, B]$.
\end{enumerate}
\end{lemma}

\begin{remark} \label{rem::biparite1}
In fact, the version of this lemma appearing in~\cite{CNP} is only for odd $n$, whereas if $G$ is bipartite and Hamiltonian it has to be balanced, rendering $n$ even. However, it is not hard to see that its proof in~\cite{CNP} applies essentially verbatim to the bipartite case as well. Indeed, the only place where the assumption that $n$ is odd is used in~\cite{CNP} is in the proof of Property \emph{(C1)}; if $n$ is odd, then it is an immediate corollary of Property \emph{(C2)}. On the other hand, if $G$ (and thus also $R$) is bipartite, then \emph{(C1)} is an immediate corollary of Property \emph{(C3)} from the same lemma. 
\end{remark}

The following definition of a so-called \emph{parity switcher} is central to the method of~\cite{CNP}. It describes a construction that, given graphs $G$ and $R$ as in Lemma~\ref{lem::subgraphR}, aids one in finding a Hamilton cycle of $G$ with an odd number of edges in $R$, thus arriving at a contradiction to (C2) above. 

\begin{definition} \label{def::paritySwitcher}
Given a graph $G$ and a subgraph $R \subseteq G$, a subgraph $W \subseteq G$ is called an $R$-\emph{parity-switcher} if it consists of an even cycle $C = (v_1, v_2, \ldots, v_{2k}, v_1)$ with an odd number of edges in $R$, and vertex-disjoint paths $P_2, \ldots, P_k$ such that for every $2 \leq i \leq k$, the endpoints of $P_i$ are $v_i$ and $v_{2k-i+2}$ and $V(P_i) \cap V(C) = \{v_i, v_{2k-i+2}\}$.
\end{definition}

In justification of its title, the parity switcher, as per Definition~\ref{def::paritySwitcher}, contains two Hamilton paths whose endpoints are $v_1$ and $v_{k+1}$; one having an odd number of edges of $R$ and the other having an even number of edges of $R$ (additional details and a figure can be found in~\cite{CNP}).

We may now specify the recipe from~\cite{CNP}; it consists of the following five steps.
\begin{enumerate}

\item [(S1)] Let $G$ be an $n$-vertex Hamiltonian graph, where $n$ is odd or $G$ is bipartite. Suppose it satisfies $\mathcal{C}_n(G) \neq \mathcal{C}(G)$, and let $R \subseteq G$ be a subgraph as in Lemma~\ref{lem::subgraphR}.

\item [(S2)] Find in $G$ a (small) $R$-parity-switcher $W$, that is,
\begin{enumerate}
\item [(S2a)] Find an even (short) cycle $C = (v_1, \ldots, v_{2k}, v_1)$ with an odd number of edges in $R$.

\item [(S2b)] Find pairwise vertex-disjoint (short) paths $P_2, \ldots, P_k$ such that, for every $2 \leq i \leq k$, the endpoints of $P_i$ are $v_i$ and $v_{2k-i+2}$, and $V(P_i) \cap V(C) = \{v_i, v_{2k-i+2}\}$.
\end{enumerate}

\item [(S3)] Find in $(G \setminus V(W)) \cup \{v_1, v_{k+1}\}$ a Hamilton path $P$ whose endpoints are $v_1$ and $v_{k+1}$.

\item [(S4)] If $P$ contains an odd (even) number of edges of $R$, then choose a Hamilton path $P'$ in $W$ whose endpoints are $v_1$ and $v_{k+1}$  with an even (odd) number of edges of $R$.

\item [(S5)] Conclude that the concatenation of $P$ and $P'$ yields a Hamilton cycle $H \subseteq G$ with an odd number of edges in $R$, contradicting (C2).

\end{enumerate}

Note that there is nothing to prove in steps (S4) and (S5). Moreover, whenever we start with a graph which we know to be Hamiltonian, Step (S1) becomes immediate. The main task is thus to deal with steps (S2) and (S3). 

\begin{remark} \label{rem::biparite}
Again, the method of Christoph, Nenadov, and Petrova~\cite{CNP} was only designed for odd $n$ (see the description of Step \emph{(S1)}). However, as noted in Remark~\ref{rem::biparite1}, it applies to the bipartite case as well (the only change is Lemma~\ref{lem::subgraphR}). 
\end{remark}


\section{Highly connected graphs with a small independence number} \label{sec::CEcycleSpace}
In this section we prove Theorems~\ref{th::CycleSpaceCE}, \ref{th::CycleSpaceCELinearDegree}, and~\ref{th::CycleSpaceCDS}. We begin by recalling several known results that facilitate our proofs.

The first such result, due to Chv\'atal and Erd\H{o}s, asserts that a slight strengthening of the condition appearing in Theorem~\ref{th::HamCE} strengthens its conclusion.
\begin{theorem} [\cite{CE}] \label{th::HamiltonConnectCE}
If $G$ is a graph satisfying $\kappa(G) > \alpha(G)$, then $G$ is Hamilton-connected.
\end{theorem}

Improving earlier results by Robertson and Seymour~\cite{RS} and by Bollob\'as and Thomason~\cite{BT}, it was proved by Thomas and Wollan~\cite{TW} that highly connected graphs are also highly \emph{linked}.

\begin{theorem} [\cite{TW}] \label{th::BT}
Let $G$ be a graph and let $x_1, y_1, \ldots, x_r, y_r$ be $2r$ distinct vertices of $G$. If $\kappa(G) \geq 10 r$, then $G$ admits pairwise vertex-disjoint paths $P_1, \ldots, P_r$ such that, for every $i \in [r]$, the endpoints of $P_i$ are $x_i$ and $y_i$.
\end{theorem}

A variant of Theorem~\ref{th::BT} in which the paths span the vertex-set of $G$ was proved by Aigner-Horev and the authors.

\begin{theorem} [\cite{AHK}] \label{lem::disjointPaths}
Let $G = (V,E)$ be an $n$-vertex graph and let $x_1, y_1, \ldots, x_r, y_r$ be $2r$ distinct vertices of $G$. There exists a constant $c > 0$ such that if $\kappa(G) \geq c \max \{\alpha(G), \log n, r\}$, then $G$ admits pairwise vertex-disjoint paths $P_1, \ldots, P_r$ such that $V(P_1) \cup \ldots \cup V(P_r) = V$ and, for every $i \in [r]$, the endpoints of $P_i$ are $x_i$ and $y_i$.
\end{theorem}

The following result allows one to partition a graph into large highly connected subgraphs.

\begin{lemma} [\cite{BFKM}] \label{lem::highConnectivity}
Let $H = (V,E)$ be an $n$-vertex graph with minimum degree $k > 0$. Then, there exists a partition $V = V_1 \cup \ldots \cup V_t$ such that, for every $i \in [t]$, the subgraph $H[V_i]$ is $k^2/(16 n)$-connected and $|V_i| \geq k/8$. 
\end{lemma}

The following simple observation will be useful in the proof of Theorem~\ref{th::CycleSpaceCELinearDegree}.
\begin{observation} \label{obs::ConnCollapsedGraph}
Let $G$ be a $k$-edge-connected graph and let $V_1 \cup \ldots \cup V_t$ be a partition of $V(G)$ into non-empty parts. Let $H$ be the multigraph obtained from $G$ by collapsing each set $V_i$ into a single vertex, that is,  $V(H) = \{v_1, \ldots, v_t\}$ and for every $1 \leq i < j \leq t$ there are $e_G(V_i, V_j)$ edges of $H$ between $v_i$ and $v_j$. Then, $H$ is $k$-edge-connected. 
\end{observation}

\begin{proof}
Fix any $1 \leq i < j \leq t$, and let $x \in V_i$ and $y \in V_j$ be arbitrary vertices. Since $G$ is $k$-edge-connected, it admits $k$ pairwise edge-disjoint paths between $x$ and $y$. Upon collapsing $G$ to yield $H$, each such path becomes a walk between $v_i$ and $v_j$ and these walks are pairwise edge-disjoint. Each such walk contains a path between $v_i$ and $v_j$; clearly these paths are also pairwise edge-disjoint. Since $i$ and $j$ were arbitrary, it follows that $H$ admits $k$ pairwise edge-disjoint paths between $v_i$ and $v_j$ for every $1 \leq i < j \leq t$, and is thus $k$-edge-connected.     
\end{proof}

The following result is the main technical tool of this section; it handles Step (S2a).

\begin{lemma} \label{lem::CECycleLemma}
Let $G$ be a graph satisfying $\kappa(G) \geq 12 \alpha(G) + 8$. Let $R$ be a subgraph of $G$ as in Lemma~\ref{lem::subgraphR}. Then, there exists a cycle $C \subseteq G$ for which the following two  properties hold.
\begin{enumerate}
\item [\emph{(a)}] $|C|$ is even and $|E(C) \setminus E(R)|$ is odd;

\item [\emph{(b)}] $|C| \leq 10 \alpha(G) + 8$.
\end{enumerate}
\end{lemma}


\begin{proof}
Let $k = \kappa(G)$ and let $\alpha = \alpha(G)$. It follows by Property (C3) from Lemma~\ref{lem::subgraphR} that $R$ is $k/2$-edge-connected and, in particular, that $\delta(R) \geq k/2$. We begin by proving that there exists a cycle of some length satisfying Property (a); we distinguish between the following two cases.
\begin{enumerate}
\item [(1)] $R$ is bipartite. Let $A \cup B$ be the bipartition of $R$. Since $R \neq G[A, B]$ holds by Property (C3) from Lemma~\ref{lem::subgraphR}, there exist vertices $x \in A$ and $y \in B$ such that $xy \in E(G) \setminus E(R)$. Since $R$ is connected, there is a path $P$ in $R$ between $x$ and $y$. Since $R$ is bipartite, the length of $P$ is odd. Hence, combined with the edge $xy$, this yields a cycle which satisfies Property (a). 

\item [(2)] $R$ is not bipartite. Let $C_1 = (x_1, \ldots, x_{2t-1}, x_1)$ be an odd cycle in $R$ of minimum possible length and let $U = V(G) \setminus V(C_1)$. Note that we may assume that $|V(C_1)| \leq 2 \alpha + 1$. Indeed, assume that $|V(C_1)| \geq 2 \alpha + 3$ and let $S = \{x_{2i} : 1 \leq i \leq \alpha + 1\}$. Since $|S| \geq \alpha + 1$, there exist $1 \leq i < j \leq \alpha + 1$ such that $x_{2i} x_{2j} \in E(G)$. By the assumed minimality of $C_1$, it must hold that $x_{2i} x_{2j} \in E(G) \setminus E(R)$. In this case, one of the two cycles in $C_1 \cup \{x_{2i} x_{2j}\}$ which contain the edge $x_{2i} x_{2j}$ satisfies Property (a). This case is further divided into the following two subcases.
\begin{enumerate}
\item [(2.1)] $(G \setminus R)[U]$ is not bipartite. Let $C_2 = (y_1, \ldots, y_{2s-1}, y_1)$ be an odd cycle in $(G \setminus R)[U]$. Since $G$ is 2-connected, it follows by Menger's Theorem (see, e.g., Exercise 4.2.28 in~\cite{West}) that there are distinct vertices $u, u' \in V(C_1)$ and $v, v' \in V(C_2)$ for which there are vertex-disjoint paths $P_1$ and $P_2$ such that the endpoints of $P_1$ are $u$ and $v$, the endpoints of $P_2$ are $u'$ and $v'$, $V(C_1) \cap V(P_1) = \{u\}$, $V(C_1) \cap V(P_2) = \{u'\}$, $V(C_2) \cap V(P_1) = \{v\}$, and $V(C_2) \cap V(P_2) = \{v'\}$. By choosing one of the two paths that connect $u$ and $u'$ along $C_1$, which we denote by $P$, one can ensure that $|E(R) \cap (E(P) \cup E(P_1) \cup E(P_2))|$ is odd. Similarly, by choosing one of the two paths that connect $v$ and $v'$ along $C_2$, which we denote by $P'$, one can ensure that $|E(G \setminus R) \cap (E(P) \cup E(P_1) \cup E(P_2) \cup E(P'))|$ is odd. In particular, the cycle $P \cup P_1 \cup P' \cup P_2$ is even and thus satisfies Property (a). 

\item [(2.2)] $(G \setminus R)[U]$ is bipartite. Let $U = A \cup B$ be some bipartition of $(G \setminus R)[U]$. Let $P_1 = (u_1, \ldots, u_{4\alpha+1})$ be an arbitrary path in $R[U]$; such a path exists since $\delta(R[U]) \geq \delta(R) - |V(C_1)| \geq k/2 - (2 \alpha + 1) \geq 4 \alpha + 1$. Let $G' = G[(V(G) \setminus V(P_1)) \cup \{u_1, u_{4\alpha+1}\}]$ and note that $\kappa(G') \geq \kappa(G) - (4 \alpha - 1) > \alpha(G) \geq \alpha(G')$. It thus follows by Theorem~\ref{th::HamiltonConnectCE} that $G'$ is Hamilton-connected. Let $P_2$ be a Hamilton path of $G'$ whose endpoints are $u_1$ and $u_{4\alpha+1}$ and let $C'$ be the Hamilton cycle of $G$ that is obtained via the concatenation of $P_1$ and $P_2$. Note that $C'$ must contain an even number of edges of $R$ (and thus an odd, and in particular positive, number of edges of $G \setminus R$) as otherwise we would obtain a contradiction to Property (C2) from Lemma~\ref{lem::subgraphR}. Assume without loss of generality that $|A \cap V(P_1)| \geq |B \cap V(P_1)|$. Let $u_{i_1}, \ldots, u_{i_r}$ be the vertices of $A \cap V(P_1)$, where $r \geq 2 \alpha + 1$. Let $S_1 = \{u_{i_j} : 1 \leq j \leq r \textrm{ and } i_j \textrm{ is even}\}$ and let $S_2 = \{u_{i_j} : 1 \leq j \leq r \textrm{ and } i_j \textrm{ is odd}\}$; assume without loss of generality that $|S_1| \geq |S_2|$. Note that $S_1$ is an independent set in $C'$ of size $|S_1| \geq \alpha + 1$. Hence, there exist vertices $u_{i_s}$ and $u_{i_t}$ in $S_1$ such that $u_{i_s} u_{i_t} \in E(G)$. Since $S_1 \subseteq A$ by construction, it follows that $u_{i_s} u_{i_t} \in E(R)$. Let $P'$ be the subpath of $P_1$ whose endpoints are $u_{i_s}$ and $u_{i_t}$. Since both $i_s$ and $i_t$ are even, we conclude that $(C' \setminus P') \cup \{u_{i_s} u_{i_t}\}$ is a cycle satisfying Property (a).
\end{enumerate}
\end{enumerate} 

Let $C = (v_1, \ldots, v_{2r}, v_1)$ be a cycle of minimum length that satisfies Property (a); we claim that it satisfies Property (b) as well. Suppose for a contradiction that $|C| \geq 10 \alpha + 10$. For every $1 \leq i \leq 5$ let $A_i = \{v_{2j} : (i-1) \alpha + i \leq j \leq i \alpha + i\}$. It follows by the definition of $\alpha$ that for every $1 \leq i \leq 5$ there are indices $s_i, t_i \in A_i$ such that $s_i < t_i$ and $v_{s_i} v_{t_i} \in E(G)$. By the pigeonhole principle, at least three of these edges are in $E(R)$ or at least three are in $E(G) \setminus E(R)$. Assume without loss of generality that $v_{s_i} v_{t_i} \in E(R)$ for every $1 \leq i \leq 3$ (the remaining cases are analogous and can be treated similarly). For every $1 \leq i \leq 3$ let $P_i = (v_{s_i}, v_{s_i+1}, \ldots, v_{t_i})$. The path $P_i$ is said to be \emph{even} if $|E(P_i) \cap E(R)|$ is even; otherwise $P_i$ is said to be \emph{odd}. Another application of the pigeonhole principle implies that at least two of $P_1, P_2$, and $P_3$ are even or at least two of them are odd. Assume without loss of generality that $P_1$ and $P_2$ are both even (again, the remaining cases are analogous and can be treated similarly). It then follows that $(C \cup \{v_{s_1} v_{t_1}, v_{s_2} v_{t_2}\}) \setminus (P_1 \cup P_2)$ forms a cycle satisfying Property (a), contrary to the minimality of $C$.    
\end{proof} 

We are now in a position to prove Theorems~\ref{th::CycleSpaceCE}, \ref{th::CycleSpaceCELinearDegree}, and~\ref{th::CycleSpaceCDS}.

\begin{proof} [Proof of Theorem~\ref{th::CycleSpaceCE}]
Let $c'$ be the constant whose existence is ensured by Theorem~\ref{lem::disjointPaths} and let $c = \max\{42, 20 c'\}$. Let $\alpha = \alpha(G)$ and assume that $\kappa(G) \geq c \min \{\max \{\alpha, \log n\}, \alpha^2\}$.

Suppose for a contradiction that $\mathcal{C}_n(G) \neq \mathcal{C}(G)$. We follow the recipe that was presented in Section~\ref{sec::recipe}. That is, we need to handle steps (S1), (S2), and (S3). Our assumption that $\mathcal{C}_n(G) \neq \mathcal{C}(G)$ will then lead to the contradiction appearing in (S5).  

It follows by Theorem~\ref{th::HamCE} that $G$ is Hamiltonian. Let $R$ be as in the premise of Lemma~\ref{lem::subgraphR}; this takes care of (S1). 

Next, we take care of (S2a). It follows by Lemma~\ref{lem::CECycleLemma} that $G$ contains an even cycle $C = (v_1, \ldots, v_{2r}, v_1)$, having an odd number of edges in $R$, whose length is at most $10 \alpha + 8 < 20 \alpha$. 

Assume first that $\kappa(G) \geq c \log n$. We may then apply Theorem~\ref{lem::disjointPaths} with $(x_1, y_1) = (v_1, v_{r+1})$ and $(x_i, y_i) = (v_i, v_{2r-i+2})$ for every $2 \leq i \leq r$ to obtain both (S2b) and (S3).

Assume then that $\kappa(G) \geq c \alpha^2$. We begin by handling Step (S2b). Let $G_1 = G \setminus \{v_1, v_{r+1}\}$ and note that $\kappa(G_1) \geq \kappa(G) - 2 \geq 10 r$. Applying Theorem~\ref{th::BT} to $G_1$ with $(x_{i-1}, y_{i-1}) = (v_i, v_{2r-i+2})$ for every $2 \leq i \leq r$ yields pairwise vertex-disjoint paths $P_2, \ldots, P_r$ such that for every $2 \leq i \leq r$, the endpoints of $P_i$ are $v_i$ and $v_{2r-i+2}$. Choose such a path system for which $\max \{|V(P_i)| : 2 \leq i \leq r\}$ is minimal and note that this implies that $\max \{|V(P_i)| : 2 \leq i \leq r\} \leq 2 \alpha$. Indeed, suppose for a contradiction that there exists some $2 \leq i \leq r$ such that $P_i = (x_1, x_2, \ldots, x_t)$ for some $t \geq 2 \alpha + 1$. Let $S = \{x_i : 1 \leq i \leq t \textrm{ is odd}\}$, and note that $|S| \geq \alpha + 1$. It follows that there exist indices $1 \leq i < j \leq t$ such that $x_i x_j \in E(G)$, yielding the shorter path $P'_i = (x_1, \ldots, x_i, x_j, \ldots, x_t)$.

Finally, we take care of (S3). Let $W = V(P_2) \cup \ldots \cup V(P_r)$ and let $G' = G \setminus W$. Note that 
$$
\kappa(G') \geq \kappa(G) - |W| \geq c \alpha^2 - 40 \alpha^2 > \alpha \geq \alpha(G'),
$$
and thus $G'$ is Hamilton-connected by Theorem~\ref{th::HamiltonConnectCE}. We conclude that $G'$ admits a Hamilton path whose endpoints are $v_1$ and $v_{r+1}$.
\end{proof}

\begin{proof} [Proof of Theorem~\ref{th::CycleSpaceCELinearDegree}]
Let $\alpha = \alpha(G)$ and assume that $\delta(G) \geq \delta n$ and $\kappa(G) \geq \max \left\{180 \alpha, 2000 \delta^{-2} \right\}$. We may assume that $\alpha = O(\log n)$ and that $\delta = \Omega \left(1/\sqrt{\log n} \right)$ as otherwise Theorem~\ref{th::CycleSpaceCELinearDegree} is an immediate corollary of Theorem~\ref{th::CycleSpaceCE}. 

Suppose for a contradiction that $\mathcal{C}_n(G) \neq \mathcal{C}(G)$. We follow the recipe that was presented in Section~\ref{sec::recipe}. That is, we need to handle steps (S1), (S2), and (S3). Our assumption that $\mathcal{C}_n(G) \neq \mathcal{C}(G)$ will then lead to the contradiction appearing in (S5).  

Since $\kappa(G) \geq 180 \alpha \geq \alpha$, it follows by Theorem~\ref{th::HamCE} that $G$ is Hamiltonian. Let $R$ be as in the premise of Lemma~\ref{lem::subgraphR}; this takes care of (S1). 

Next, we take care of (S2a). Since $\kappa(G) \geq 12 \alpha + 8$, it follows by Lemma~\ref{lem::CECycleLemma} that $G$ contains an even cycle $C = (v_1, \ldots, v_{2k}, v_1)$, having an odd number of edges in $R$, whose length is at most $10 \alpha + 8$.


In preparation for Step (S3), let $G_1 = G \setminus V(C)$; note that $\delta(G_1) \geq \delta(G) - |V(C)| \geq \delta n/2$. Applying Lemma~\ref{lem::highConnectivity} to $G_1$ we obtain a partition $V(G_1) = V_1 \cup \ldots \cup V_t$ such that $\kappa(G_1[V_i]) \geq \delta^2 n/64$ and $|V_i| \geq \delta n/16$ hold for every $1 \leq i \leq t$; in particular $t \leq 16/\delta$. Let $x \in N_G(v_1, V_r)$ for some $1 \leq r \leq t$ and let $y \in N_G(v_{k+1}, V_s \setminus \{x\})$ for some $1 \leq s \leq t$; such vertices $x$ and $y$ exist since $\min \{\textrm{deg}_G(v_1), \textrm{deg}_G(v_{k+1})\} \geq \delta(G) \geq \delta n > |V(C)|$. In order to handle Step (S3) we wish to set aside a certain structure; this is done via the following claim.
\begin{claim} \label{cl::HamWalkFewEdges}
There is a sequence $(u_1, w_1, u_2, w_2, \ldots, u_q, w_q)$ of (not necessarily distinct) vertices of $G_1$ satisfying all of the following properties.
\begin{enumerate}
\item [\emph{(1)}] $u_1 \in V_r$ and $w_q \in V_s$;

\item [\emph{(2)}] $u_i, w_i \in V_1 \cup \ldots \cup V_t \setminus \{x,y\}$ for every $1 \leq i \leq q$;

\item [\emph{(3)}] $u_\ell w_\ell \in \bigcup_{1 \leq i < j \leq t} E_G(V_i, V_j) $ for every $1 \leq \ell \leq q$;

\item [\emph{(4)}] Setting $w_0 = x$ and $u_{q+1} = y$, for every $1 \leq i \leq t$ it holds that $|V_i \cap \{w_0, u_1, w_1, u_2,  \ldots, w_q, u_{q+1}\}| \leq 2 t$;

\item [\emph{(5)}] For every $1 \leq i \leq t$ and every $0 \leq j \leq q$ it holds that $w_j \in V_i$ if and only if $u_{j+1} \in V_i$. Moreover, for every $1 \leq i \leq t$ there exists some $0 \leq j \leq q$ such that $w_j \in V_i$, $u_{j+1} \in V_i$, and $w_j \neq u_{j+1}$;

\item [\emph{(6)}] If two vertices of the sequence $(u_1, w_1, u_2, w_2, \ldots, u_q, w_q)$ are equal, then there exists some $1 \leq j \leq q-1$ such that these two vertices are $w_j$ and $u_{j+1}$.
\end{enumerate}
\end{claim}


\begin{proof} 
For convenience assume that $r=1$ and $s=t$ (this is possible by relabeling the indices unless $r = s$; the latter case, however, can be treated similarly). We construct the required sequence in $t-1$ steps, that is, for every $1 \leq j \leq t-1$ we construct a subsequence $S_j := (u_{t_j+1}, w_{t_j+1}, \ldots, u_{t_j+q_j}, w_{t_j+q_j})$, where $\sum_{i=1}^{t-1} q_i = q$ and $t_j = \sum_{i=1}^{j-1} q_i$ for every $1 \leq j \leq t-1$ (in particular, $t_1 = 0$), satisfying the following two properties 
\begin{enumerate}

\item [(i)] $u_{t_j+1} \in V_j$ and $w_{t_j+q_j} \in V_{j+1}$; 

\item [(ii)] $|A_j \cap V_i| \leq 2$ holds for every $1 \leq i \leq t$, where $A_j := \{u_{t_j+1}, w_{t_j+1}, \ldots, u_{t_j+q_j}, w_{t_j+q_j}\}$;

\item [(iii)] $A_j \cap (A_1 \cup \ldots \cup A_{j-1} \cup \{x,y\}) = \varnothing$.
\end{enumerate}
Suppose that for some $1 \leq j \leq t-1$ we have already constructed $S_1, \ldots, S_{j-1}$ for which the above properties hold, and now wish to construct $S_j$. Let $\tilde{G}_j = G_1 \setminus (A_1 \cup \ldots \cup A_{j-1} \cup \{x,y\})$ and note that     
\begin{align} \label{eq::ConGj}
\kappa(\tilde{G}_j) &\geq \kappa(G) - |V(C)| - \sum_{i-1}^{j-1} |A_i| - 2 \geq \max \left\{180 \alpha, 2000 \delta^{-2} \right\} - (10 \alpha + 10) - 2 t^2 \geq 1,
\end{align}
where the second inequality holds by Lemma~\ref{lem::CECycleLemma}(b) and by (ii).

For every $1 \leq i \leq t$, let $V_i^j = V_i \setminus (A_1 \cup \ldots \cup A_{j-1} \cup \{x,y\})$, and let $H_j$ be the multigraph obtained from $\tilde{G}_j$ by collapsing each $V_i^j$ into a single vertex $v_i$. It follows by~\eqref{eq::ConGj} and by Observation~\ref{obs::ConnCollapsedGraph} that $H_j$ is connected. In particular, $H_j$ admits a path $P_j$ between $v_j$ and $v_{j+1}$; clearly the length of $P_j$ is at most $t-1$. Viewing $P_j$ as a directed path from $v_j$ to $v_{j+1}$, let $e_1, \ldots, e_{q_j}$ be the edges of $P_j$ appearing in this order. Viewing $e_1, \ldots, e_{q_j}$ as edges of $\tilde{G}_j$, for every $1 \leq i \leq q_j$, let $u_{t_j+i}$ be the tail of $e_i$ in $\tilde{G}_j$ and let $w_{t_j+i}$ be the head of $e_i$ in $\tilde{G}_j$ (note that it is possible that $w_{t_j+i} = u_{t_j+i+1}$ for some values of $i$). It is straightforward to verify that $S_j := (u_{t_j+1}, w_{t_j+1}, \ldots, u_{t_j+q_j}, w_{t_j+q_j})$ satisfies properties (i), (ii), and (iii). 

Having constructed $S_1, \ldots, S_{t-1}$ in this manner, it is then straightforward to verify that their concatenation, namely $(u_1, w_1, u_2, w_2, \ldots, u_q, w_q)$, satisfies properties (1)--(6).
\end{proof}

Returning to the proof of Theorem~\ref{th::CycleSpaceCELinearDegree}, fix a sequence $(u_1, w_1, u_2, w_2, \ldots, u_q, w_q)$ as per Claim~\ref{cl::HamWalkFewEdges}, and let $G_2 = G \setminus \{u_1, w_1, u_2, w_2, \ldots, u_q, w_q, x, y, v_1, v_{k+1}\}$.   

Next, we handle Step (S2b). Note that 
\begin{align*}
\kappa(G_2) &\geq \kappa(G) - (2q+4) \geq \max \left\{180 \alpha, 2000 \delta^{-2} \right\} - 2 t^2 - 2 \\ 
&\geq \max \left\{180 \alpha, 2000 \delta^{-2} \right\} - 514 \delta^{-2} \geq 90 \alpha \geq 10 k,
\end{align*}
where the second inequality holds by Claim~\ref{cl::HamWalkFewEdges}(4). Applying Theorem~\ref{th::BT} to $G_2$ with $(x_{i-1}, y_{i-1}) = (v_i, v_{2k-i+2})$ for every $2 \leq i \leq k$ yields pairwise vertex-disjoint paths $P_2, \ldots, P_k$ such that for every $2 \leq i \leq k$, the endpoints of $P_i$ are $v_i$ and $v_{2k-i+2}$. Moreover, as in the proof of Theorem~\ref{th::CycleSpaceCE}, we choose the paths $P_2, \ldots, P_k$ such that $\max \{|V(P_i)| : 2 \leq i \leq k\} \leq 2 \alpha$ holds.

Finally, we take care of Step (S3). Let $W = V(P_2) \cup \ldots \cup V(P_k)$ and let $G_3 = G \setminus W$. Note that $\{u_1, w_1, u_2, w_2, \ldots, u_q, w_q, x, y, v_1, v_{k+1}\} \subseteq V(G_3)$ and that
\begin{align} \label{eq::HighConnectivity}
\kappa(G_3[V_i \setminus W]) \geq \kappa(G_1[V_i]) - |W| \geq \delta^2 n/64 - 2 \alpha (5 \alpha + 4) \geq \delta^2 n/65
\end{align}
holds for every $1 \leq i \leq t$. Set $w_0 = x$ and $u_{q+1} = y$. For every $1 \leq i \leq t$ let $B_i = \{1 \leq j \leq q+1 : u_j \in V_i\}$, and let $b_i = |B_i|$; note that $b_i \leq t$ holds for every $1 \leq i \leq t$ by properties (4) and (5) from Claim~\ref{cl::HamWalkFewEdges}. For every $1 \leq i \leq t$, it then  follows by~\eqref{eq::HighConnectivity} and by Theorem~\ref{lem::disjointPaths} that there are pairwise vertex-disjoint paths $P^i_1, \ldots, P^i_{b_i}$ such that $V(P^i_1) \cup \ldots \cup V(P^i_{b_i}) = V_i \setminus W$ and for every $1 \leq j \leq b_i$ the endpoints of $P^i_j$ are $w_{j-1}$ and $u_j$ (such an application of Theorem~\ref{lem::disjointPaths} is possible by Property (5) of Claim~\ref{cl::HamWalkFewEdges} which ensures that $w_{j-1} \neq u_j$ holds for at least one $j \in B_i$. Note that whenever $w_{j-1} = u_j$, the path $P^i_j$ consists of a single vertex). Combining the paths in $\{P^i_j : 1 \leq i \leq t, j \in B_i\}$ and the edges in $\{u_i w_i : 1 \leq i \leq q\} \cup \{v_1 x, v_{k+1} y\}$ yields a Hamilton path of $G_3$ whose endpoints are $v_1$ and $v_{k+1}$.
\end{proof}


\begin{proof} [Proof of Theorem~\ref{th::CycleSpaceCDS}]
Let $\alpha = \alpha(G)$ and assume that $S_1, \ldots, S_t$ are pairwise disjoint CDSs of $G$, where $t \geq 16 \alpha + 12$. Note that $\kappa(G) \geq t \geq 16 \alpha + 12$.

Suppose for a contradiction that $\mathcal{C}_n(G) \neq \mathcal{C}(G)$. We follow the recipe that was presented in Section~\ref{sec::recipe}. That is, we need to handle steps (S1), (S2), and (S3). Our assumption that $\mathcal{C}_n(G) \neq \mathcal{C}(G)$ will then lead to the contradiction appearing in (S5).  

Since $\kappa(G) \geq 16 \alpha + 12 \geq \alpha$, it follows by Theorem~\ref{th::HamCE} that $G$ is Hamiltonian. Let $R$ be as in the premise of Lemma~\ref{lem::subgraphR}; this takes care of (S1). 

Next, we take care of (S2a). Since $\kappa(G) \geq 16 \alpha + 12$, it follows by Lemma~\ref{lem::CECycleLemma} that $G$ contains an even cycle $C = (v_1, \ldots, v_{2r}, v_1)$, having an odd number of edges in $R$, whose length is at most $10 \alpha + 8$. 

Next, we handle Step (S2b). Let $I = \{i \in [t] : S_i \cap V(C) \neq \varnothing\}$ and note that $|I| \leq |C| \leq 10 \alpha + 8$. Let $\varphi : \{2, \ldots, r\} \to [t] \setminus I$ be an injective mapping; such a mapping exists since $t \geq 16 \alpha + 12 \geq |I| + r-1$. For every $2 \leq i \leq r$ let $P_i$ be a path in $G$ whose endpoints are $v_i$ and $v_{2r-i+2}$ and such that $V(P_i) \setminus \{v_i, v_{2r-i+2}\} \subseteq S_{\varphi(i)}$; such a path exists since $S_{\varphi(i)}$ is a dominating set and $G[S_{\varphi(i)}]$ is connected.

 


Finally, we take care of (S3). Let $W = V(P_2) \cup \ldots \cup V(P_r)$ and let $G' = G \setminus W$. Let $J = \{i \in [t] : S_i \cap W = \varnothing\}$, and note that $|J| \geq t - |C| - (r-1) > \alpha$. Since $\bigcup_{i \in J} S_i \subseteq V(G')$, it follows that $\kappa(G') \geq |J| > \alpha \geq \alpha(G')$, and thus $G'$ is Hamilton-connected by Theorem~\ref{th::HamiltonConnectCE}. We conclude that $G'$ admits a Hamilton path whose endpoints are $v_1$ and $v_{r+1}$.
\end{proof}

We end this section with a construction which proves Proposition~\ref{prop::CElowerbound}, Proposition~\ref{prop::CDSlowerbound}, and Proposition~\ref{prop::Erdoslowerbound}.
\begin{proof} [Proof of Propositions~\ref{prop::CElowerbound}, \ref{prop::CDSlowerbound}, and~\ref{prop::Erdoslowerbound}]
Let $k \geq 2$ and $n \geq 3k-1$ be integers, where $n$ is odd (note that $n$ may be taken to be larger than $f(k)$ for any function $f$). Let $A \cup X \cup I$ be a partition of $[n]$, where $|X| = k$, $|I| = 2(k-1)$, and $|A| = n - (3k-2)$; note that $A \neq \varnothing$ and $|A \cup X| \geq 3$. Let $G = ([n], E)$ be the graph with edge set $E = \{xy : x \neq y \in A \cup X\} \cup \{xy : x \in X, y \in I\} \cup M$, where $M = \{x_i y_i : 1 \leq i \leq k-1\}$ is a perfect matching of $G[I]$. It is easy to see that $X$ is a minimum cut in $G$ implying that $\kappa(G) = |X| = k$. Similarly, it is easy to see that any maximum independent set of $G$ contains one endpoint of every edge of $M$ and one vertex of $A$ implying that $\alpha(G) = k$. It follows that $\kappa(G) \geq \alpha(G)$ and thus $G$ is Hamiltonian by Theorem~\ref{th::HamCE} (it is also easy to simply construct a Hamilton cycle of $G$). Moreover, since every vertex of $X$ is a CDS, $G$ admits $\alpha(G)$ pairwise disjoint CDSs. It remains to prove that $\mathcal{C}_n(G) \neq \mathcal{C}(G)$. In order to do so we first prove that every Hamilton cycle of $G$ includes every edge of $M$. Indeed, suppose for a contradiction that there exist some $1 \leq i \leq k-1$ and some Hamilton cycle $C$ of $G$ such that $x_i y_i \notin E(C)$. It follows that $C$ is a Hamlton cycle of $G' := G \setminus \{x_i y_i\}$ and thus, in particular, $G'$ is 1-tough. This, however, is not true since removing the $k$ vertices of $X$ from $G'$ creates $k+1$ connected components, namely $G[A]$, $x_i$, $y_i$, and $x_j y_j$ for every $j \in [k-1] \setminus \{i\}$. Having proved this claim, let $T$ be a triangle in $G[A \cup X]$. Since $T$ is an odd cycle, in order to generate it we need an odd number of Hamilton cycles. On the other hand, since $E(T) \cap M = \varnothing$ and every Hamilton cycle of $G$ includes all edges of $M$, in order to generate $T$ we need an even number of Hamilton cycles. We conclude that $T \in \mathcal{C}(G) \setminus \mathcal{C}_n(G)$.  
\end{proof}

\section{Graphs with a small bipartite independence number} \label{sec::BIcycleSpace} 
In this section we prove Theorem~\ref{th::CycleSpaceBI}. We begin by stating and proving several auxiliary results that facilitate our proof. The following known result takes care of Step (S3).


 \begin{theorem} [Theorem 6 in~\cite{ZBWL}] \label{th::HamConMY}
Let $G$ be a graph on at least 3 vertices. If $\delta(G) > \tilde{\alpha}(G)$, then $G$ is Hamilton-connected.
\end{theorem}

The following result takes care of Step (S2a).
\begin{lemma} \label{lem::BICycleLemma}
Let $G$ be a graph satisfying $\delta(G) \geq 2 \tilde{\alpha}(G) + 9$. Let $R \neq G$ be a subgraph of $G$ such that $e_R(A, B) \geq e_G(A, B)/2$ and $R \neq G[A, B]$ hold for every partition $V(G) = A \cup B$. Then, there exists a cycle $C \subseteq G$ for which the following two properties hold.
\begin{enumerate}
\item [\emph{(a)}] $|C|$ is even and $|E(C) \setminus E(R)| = 1$;

\item [\emph{(b)}] $|C| \leq 12$.
\end{enumerate}
\end{lemma}


\begin{proof}
Let $k = \tilde{\alpha}(G)$ and let $a \leq b$ be positive integers such that $a + b = k + 1$ and $E_G(A,B) \neq \emptyset$ whenever $A \subseteq V(G)$ and $B \subseteq V(G) \setminus A$ are sets of sizes $|A| = a$ and $|B| = b$. It follows by the premise of the lemma that $\delta(R) \geq \lceil \delta(G)/2 \rceil \geq a + b + 4$. We claim that $R$ is 2-connected. Indeed, suppose for a contradiction that there exists a vertex $v \in V(G)$ such that $R \setminus v$ is disconnected. Let $S$ be the vertex-set of a smallest connected component of $R \setminus v$ and let $T = V(R) \setminus (S \cup \{v\})$. Let $s = |S|$ and note that 
\begin{align} \label{eq::mediumDegree}
n/2 > s \geq \delta(R) > a+b.
\end{align}
Let $S_1$ and $S_2$ be disjoint subsets of $S$, where $|S_1| = a$ and $|S_2| = b$. For $i \in \{1,2\}$ let $T_i = N_G(S_i) \cap T$. It follows by the definition of $a$ and $b$ that $|T_1| \geq |T| - b + 1 \geq n - s - b$ and, similarly, $|T_2| \geq |T| - a + 1 \geq n - s - a$. Therefore $e_G(S, T) \geq e_G(S_1, T_1) + e_G(S_2, T_2) \geq |T_1| + |T_2| \geq 2 (n - s) - (a+b) > s$, where the last inequality holds by~\eqref{eq::mediumDegree}. We conclude that
$$
\deg_R(v, S) = e_R(S, T \cup \{v\}) \geq e_G(S, T \cup \{v\})/2 \geq (e_G(S, T) + \deg_R(v, S))/2 > (s + \deg_R(v, S))/2,
$$
which in turn implies the obvious contradiction $\deg_R(v, S) > s$.  


\medskip

Next, we prove that there exists a cycle satisfying Property (a); we distinguish between the following two cases.
\begin{enumerate}
\item [(1)] $R$ is bipartite. Let $A \cup B$ be the bipartition of $R$. Since $R \neq G[A, B]$ by the premise of the lemma, there exist vertices $x \in A$ and $y \in B$ such that $xy \in E(G) \setminus E(R)$. Since $R$ is connected, it admits a path $P$ whose endpoints are $x$ and $y$. Since $R$ is bipartite, the length of $P$ is odd. Hence, combined with the edge $xy$, this yields a cycle which satisfies Property (a). 

\item [(2)] $R$ is not bipartite. Let $C' = (x_1, \ldots, x_{2t-1}, x_1)$ be an odd cycle in $R$. Let $xy \in E(G) \setminus E(R)$ be an arbitrary edge; such an edge exists since $R \neq G$ holds by the premise of the lemma. Since $R$ is 2-connected, it follows by Menger's Theorem (see, e.g., Exercise 4.2.28 in~\cite{West}) that there are distinct vertices $u, u' \in V(C')$ for which there are vertex-disjoint paths $P_1$ and $P_2$ in $R$ such that the endpoints of $P_1$ are $u$ and $x$, the endpoints of $P_2$ are $u'$ and $y$, $V(C') \cap V(P_1) = \{u\}$, and $V(C') \cap V(P_2) = \{u'\}$. Note that it is possible that $V(P_1) = \{u\} = \{x\}$ or $V(P_2) = \{u'\} = \{y\}$. Combining the edge $xy$, the paths $P_1$ and $P_2$, and one of the two paths connecting $u$ and $u'$ along $C'$ yields a cycle $C$ which satisfies Property (a).






\end{enumerate}


Let $C = (v_1, \ldots, v_{2r}, v_1)$ be a cycle in $G$ of minimum length out of all the cycles that satisfy Property (a). Assume without loss of generality that $v_1 v_2$ is the unique edge of $E(C) \setminus E(R)$. We claim that $C$ satisfies Property (b) as well. Suppose for a contradiction that $|C| \geq 14$. We distinguish between the following cases.
\begin{enumerate}
\item [(1)] There exists an edge $v_i v_j \in E(R)$ which is a chord of $C$ and, moreover, $i$ and $j$ are of opposite parity. In this case the unique cycle in $C \cup \{v_i v_j\}$ which contains both $v_1 v_2$ and $v_i v_j$ satisfies Property (a), contrary to the minimality of $C$.

\item [(2)] $\textrm{deg}_R(v_2, V(C)) \geq 4$ and $\textrm{deg}_R(v_{2r}, V(C)) \geq 4$. Since we are not in Case (1), there exist $2 \leq i, j \leq r-1$ such that $v_2 v_{2i} \in E(R)$ and $v_{2r} v_{2j} \in E(R)$. If $i = j$, then $(v_1, v_2, v_{2i}, v_{2r}, v_1)$ is a cycle satisfying Property (a). If $i < j$, then $(v_1, v_2, v_{2i}, v_{2i+1}, \ldots, v_{2j}, v_{2r}, v_1)$ is a cycle satisfying Property (a). Finally, if $i > j$, then $(v_1, v_2, v_{2i}, v_{2i-1}, \ldots, v_{2j}, v_{2r},  v_1)$ is a cycle satisfying Property (a). In all three cases we arrive at a contradiction to the minimality of $C$.

\item [(3)] $\textrm{deg}_R(v_7, V(C)) \geq 7$ and $\textrm{deg}_R(v_{2r-5}, V(C)) \geq 7$ (note that $2r - 5 > 7$ holds by our assumption that $|C| \geq 14$). Since we are not in Case (1), there exist $1 \leq i \leq r-4$ such that $v_7 v_{2i+1} \in E(R)$ and $4 \leq j \leq r-1$ such that $v_{2r-5} v_{2j+1} \in E(R)$. Similarly to Case (2), the unique cycle in $C \cup \{v_7 v_{2i+1}, v_{2r-5} v_{2j+1}\}$ which contains $v_1 v_2$, $v_7 v_{2i+1}$, and $v_{2r-5} v_{2j+1}$ satisfies Property (a), contrary to the minimality of $C$.


\item [(4)] There exist $i \in \{2, 2r\}$ and $j \in \{7, 2r-5\}$ such that $\textrm{deg}_R(v_i, V(C)) \leq 3$ and $\textrm{deg}_R(v_j, V(C)) \leq 6$. Assume without loss of generality that $i=2$ and $j=7$ are such indices; the remaining cases can be treated similarly. Let $A$ be a subset of $N_R(v_7) \setminus V(C)$ of size $a$ and let $B$ be a subset of $N_R(v_2) \setminus (V(C) \cup A)$ of size $b$; such sets exist since $\delta(R) \geq a + b + 4 \geq a + 6$. By the definition of $\tilde{\alpha}(G)$ there exist $x \in A$ and $y \in B$ such that $xy \in E(G)$. If $xy \in E(R)$, then $(v_1, v_2, y, x, v_7, v_8, \ldots, v_{2r},  v_1)$ is a cycle satisfying Property (a). Otherwise, $xy \in E(G) \setminus E(R)$ and then $(v_2, \ldots, v_7, x, y, v_2)$ is a cycle satisfying Property (a). In both cases we arrive at a contradiction to the minimality of $C$.
\end{enumerate}
\end{proof}

The following result takes care of Step (S2b).

\begin{lemma} \label{lem::fewNeighbours}
Let $G$ be a graph satisfying $\delta(G) \geq \max \left\{2 \tilde{\alpha}(G) + 9, \tilde{\alpha}(G) + 15 \right\}$ and let $R$ be as in Lemma~\ref{lem::BICycleLemma}. Let $C = (v_1, \ldots, v_{2r}, v_1)$ be a cycle of minimum possible length of all cycles in $G$ satisfying properties \emph{(a)} and \emph{(b)} from Lemma~\ref{lem::BICycleLemma}. Then there exist vertex-disjoint paths $P_2, \ldots, P_r$ such that the following properties hold for every $2 \leq i \leq r$.
\begin{enumerate}
\item [\emph{(i)}] The endpoints of $P_i$ are $v_i$ and $v_{2r-i+2}$, and $V(P_i) \cap V(C) = \{v_i, v_{2r-i+2}\}$;

\item [\emph{(ii)}] The length of $P_i$ is 2 or 3;

\item [\emph{(iii)}] $\emph{deg}(v, V(P_i)) \leq 3$ holds for every $v \in V(G) \setminus (V(C) \cup V(P_2) \cup \ldots \cup V(P_r))$.
\end{enumerate}
\end{lemma}

\begin{proof}
Let $a \leq b$ be positive integers such that $a + b = \tilde{\alpha}(G) + 1$ and $E_G(A,B) \neq \emptyset$ whenever $A \subseteq V(G)$ and $B \subseteq V(G) \setminus A$ are sets of sizes $|A| = a$ and $|B| = b$. We construct the required paths $P_2, \ldots, P_r$ one by one as follows. Assume that for some $2 \leq i \leq r$ we have already built $P_2, \ldots, P_{i-1}$, each satisfying properties (i), (ii), and (iii), and now wish to construct $P_i$. Let $W_i = \left(\{v_1, \ldots, v_{2r}\} \cup V(P_2) \cup \ldots \cup V(P_{i-1}) \right) \setminus \{v_i, v_{2r-i+2}\}$. If there exists a vertex $w \in V(G) \setminus W_i$ which is a common neighbour of $v_i$ and $v_{2r-i+2}$, then we set $P_i = (v_i, w, v_{2r-i+2})$; it is evident that $P_i$ satisfies properties (i), (ii), and (iii) in this case. 

Assume then that $A := N_G(v_i) \setminus W_i$ and $B := N_G(v_{2r-i+2}) \setminus W_i$ are disjoint sets. It follows by the minimality of $|C|$ that $C$ has no chord $v_s v_t$ such that $s$ is even and $t$ is odd. In particular, $\max \{\textrm{deg}_G(v_i, V(C)), \textrm{deg}_G(v_{2r-i+2}, V(C))\} \leq r+1 \leq 7$, where the last inequality holds by Property (b) from Lemma~\ref{lem::BICycleLemma}. Moreover, it follows by Property (ii) that $\max \{\textrm{deg}_G(v_i, W_i \setminus V(C)), \textrm{deg}_G(v_{2r-i+2}, W_i \setminus V(C))\} \leq 2 (i-2) \leq 2(r-2) \leq 8$. Therefore $\min \{|A|, |B|\} \geq \delta(G) - 15 \geq \tilde{\alpha}(G) \geq b \geq a$. Hence, by the definition of $a$ and $b$ there are vertices $x \in A$ and $y \in B$ such that $xy \in E(G)$. Set $P_i = (v_i, x, y, v_{2r-i+2})$ and note that $P_i$ satisfies properties (i) and (ii). Since no vertex of $V(G) \setminus W_i$ is adjacent to both $v_i$ and $v_{2r-i+2}$, the path $P_i$ satisfies Property (iii) as well.     
\end{proof}

Having Theorem~\ref{th::HamConMY} and Lemmas~\ref{lem::BICycleLemma} and~\ref{lem::fewNeighbours} at our disposal, proving Theorem~\ref{th::CycleSpaceBI} becomes fairly straightforward. 

\begin{proof} [Proof of Theorem~\ref{th::CycleSpaceBI}]
Suppose for a contradiction that $\mathcal{C}_n(G) \neq \mathcal{C}(G)$. We follow the recipe that was presented in Section~\ref{sec::recipe}. That is, we need to handle steps (S1), (S2), and (S3). Our assumption that $\mathcal{C}_n(G) \neq \mathcal{C}(G)$ will then lead to the contradiction appearing in (S5).  

It follows by Theorem~\ref{th::HamMY} that $G$ is Hamiltonian. Let $R$ be as in the premise of Lemma~\ref{lem::subgraphR}; this takes care of (S1). 

Next, we take care of (S2). Starting with (S2a), it follows by Lemma~\ref{lem::BICycleLemma} that $G$ contains an even cycle of length at most $12$, and such that all its edges but exactly one are in $R$; let $C = (v_1, \ldots, v_{2r}, v_1)$ be such a cycle of minimum possible length. As for (S2b), the required paths $P_2, \ldots, P_r$ are attained via Lemma~\ref{lem::fewNeighbours}. 

Finally, we take care of (S3). Let $W = V(P_2) \cup \ldots \cup V(P_r)$. Let $G' = G \setminus W$. Similarly to the proof of Lemma~\ref{lem::fewNeighbours}, it follows by the minimality of $|C|$ and by Property (ii) from that lemma that $\max \{\textrm{deg}_G(v_1, W), \textrm{deg}_G(v_{r+1}, W)\} \leq r+1 + 2 (r-1) \leq 17$. Moreover, for every vertex $v \in V(G) \setminus W$, it follows by Property (iii) from Lemma~\ref{lem::fewNeighbours} that $\textrm{deg}_G(v, W) \leq 3 (r-1) \leq 15$. Hence, $\delta(G') \geq \delta(G) - 17 > \tilde{\alpha}(G) \geq \tilde{\alpha}(G')$ and thus $G'$ is Hamilton-connected by Theorem~\ref{th::HamConMY}. In particular, there is a Hamilton path of $G'$ whose endpoints are $v_1$ and $v_{r+1}$.
\end{proof}


\section{Bipartite graphs} \label{sec::bipartite}

Our aim in this section is to prove Theorem~\ref{th::BipCEcycleSpace}. We begin by stating and proving several auxiliary results that facilitate our proof. The following result takes care of Step (S2a).
\begin{lemma} \label{lem::BipCycleLemma}
Let $G$ be a balanced bipartite graph satisfying $\alpha_{\emph{BIP}}(G) \leq 2 \delta(G) - 6$. Let $R \neq G$ be a subgraph of $G$ such that $e_R(X, Y) \geq e_G(X, Y)/2$ holds for every partition $V(G) = X \cup Y$. Then, there exists a cycle $C \subseteq G$ for which the following two properties hold.
\begin{enumerate}
\item [\emph{(a)}] $|C| \in \{6,10\}$;

\item [\emph{(b)}] $|E(C) \setminus E(R)|$ is odd. 

\end{enumerate}
\end{lemma}

\begin{proof}
It follows by Theorem~\ref{th::HamBipOAR} that $G$ is connected. It then follows by the premise of the lemma that $R$ is connected as well. Let $V(G) = A \cup B$ be the bipartition of $G$. Note that $R$ is bipartite as well and $V(R) = A \cup B$ is the unique bipartition of $R$. Since $R \neq G$ holds by the premise of the lemma, there exist vertices $x \in A$ and $y \in B$ such that $xy \in E(G) \setminus E(R)$. Since $R$ is connected, it admits a path $P$ whose endpoints are $x$ and $y$. Since $R$ is bipartite, the length of $P$ is odd. Hence, combined with the edge $xy$, this yields an even cycle which satisfies Property (b). 

Let $C' = (v_1, \ldots, v_{2r}, v_1)$ be an even cycle in $G$ of minimum length out of all even cycles satisfying Property (b); assume without loss of generality that $v_1 \in A$. We claim that $|C'| \leq 8$. Prior to proving this, we show that $C'$ has no chords. Indeed, suppose for a contradiction that there exist $1 \leq i < j \leq 2r$ such that $v_i v_j$ is a chord of $C'$; note that without loss of generality this implies that $v_i \in A$ and $v_j \in B$. It is then straightforward to verify that one of the cycles in $C' \cup \{v_i v_j\}$ which contain $v_i v_j$ is even and satisfies Property (b). In all possible cases, the obtained cycle is shorter than $C'$, contrary to its assumed minimality. 


Next, we show that $|C'| \leq 8$. Suppose for a contradiction that $r \geq 5$. Let $k = \alpha_{\textrm{BIP}}(G)$ and note that $\delta(G) \geq k/2 + 3$ holds by the premise of the lemma. Since $v_1 \in A$ by assumption, it follows that $v_6 \in B$. Since $C'$ has no chords, it  then follows that $\textrm{deg}_G(v_1, B \setminus V(C')) \geq k/2 + 1$ and that $\textrm{deg}_G(v_6, A \setminus V(C')) \geq k/2 + 1$. Hence, by the definition of $\alpha_{\textrm{BIP}}(G)$, there are vertices $x \in N_G(v_1, B \setminus V(C'))$ and $y \in N_G(v_6, A \setminus V(C'))$ such that $xy \in E(G)$. Similarly as above, it is straightforward to verify that one of the cycles in $C' \cup \{v_1 x, x y, y v_6\}$ which contain $v_1 x, x y$, and $y v_6$, is an even cycle satisfying Property (b). In all possible cases, if $|C'| \geq 10$, then the obtained cycle is shorter than $C'$, contrary to its assumed minimality.

Next, using $C'$, we construct a cycle $C$ which satisfies properties (a) and (b). If $|C'| = 6$, then set $C = C'$. We consider the remaining two possible lengths of $C'$ separately. 
\begin{enumerate}
\item [(1)] $|C'| = 4$. Let $X = N_G(v_1, V(G) \setminus V(C')) \subseteq B$ and note that 
$$
|X| \geq \textrm{deg}_G(v_1) - 2 \geq k/2 + 1.
$$
Similarly, let $Y = N_G(v_3, V(G) \setminus V(C')) \subseteq B$ and note that $Y \neq \emptyset$. Let $y \in Y$ be an arbitrary vertex and let $Z = N_G(y, V(G) \setminus V(C')) \subseteq A$. Note that $|Z| \geq \textrm{deg}_G(y) - |V(C') \cap A| \geq k/2 + 1$. It then follows by the definition of $\alpha_{\textrm{BIP}}(G)$ that there exist vertices $x \in X$ and $z \in Z$ such that $xz \in E(G)$. It is straightforward to verify that either $(v_1, v_2, v_3, y, z, x, v_1)$ or $(v_1, v_4, v_3, y, z, x, v_1)$ is a cycle satisfying properties (a) and (b).

\item [(2)] $|C'| = 8$. Let $X = N_G(v_1, V(G) \setminus V(C')) \subseteq B$ and let $Y = N_G(v_2, V(G) \setminus V(C')) \subseteq A$. Since $C'$ has no chords, it follows that $\min \{|X|, |Y|\} \geq \delta(G) - 2 \geq k/2 + 1$. It then follows by the definition of $\alpha_{\textrm{BIP}}(G)$ that there exist vertices $x \in X$ and $y \in Y$ such that $xy \in E(G)$. It is straightforward to verify that either $(v_1, x, y, v_2, v_3, \ldots, v_8, v_1)$ is a cycle satisfying properties (a) and (b), or $(v_1, x, y, v_2, v_1)$ is a cycle of length 4 which satisfies Property (b); in the latter case we may apply Case (1) to obtain the required cycle.
\end{enumerate}
\end{proof}




The following result takes care of Step (S2b). 

\begin{lemma} \label{lem::fewNeighboursBip}
Let $G$ be a balanced bipartite graph satisfying $\alpha_{\emph{BIP}}(G) \leq 2 \delta(G) - 20$, and let $C = (v_1, \ldots, v_{2r}, v_1)$ be a cycle in $G$ satisfying properties \emph{(a)} and \emph{(b)} from Lemma~\ref{lem::BipCycleLemma}. Then there exist vertex-disjoint paths $P_2, \ldots, P_r$ such that the following properties hold for every $2 \leq i \leq r$.
\begin{enumerate}
\item [\emph{(i)}] The endpoints of $P_i$ are $v_i$ and $v_{2r-i+2}$, and $V(P_i) \cap V(C) = \{v_i, v_{2r-i+2}\}$;

\item [\emph{(ii)}] The length of $P_i$ is 2 or 4;

\item [\emph{(iii)}] $\emph{deg}_G(v, V(P_i)) \leq 2$ holds for every $v \in V(G) \setminus (V(C) \cup V(P_2) \cup \ldots \cup V(P_r))$.
\end{enumerate}
\end{lemma}

\begin{proof}
Let $k = \alpha_{\textrm{BIP}}(G)$ and note that $r \leq 5$ holds by Property (a) from Lemma~\ref{lem::BipCycleLemma}. We construct the required paths $P_2, \ldots, P_r$ one by one as follows. Assume that for some $2 \leq i \leq r$ we have already built $P_2, \ldots, P_{i-1}$, each satisfying properties (i), (ii), and (iii) (the latter meaning that $\textrm{deg}_G(v, V(P_j)) \leq 2$ holds for every $2 \leq j \leq i-1$ and every $v \in V(G) \setminus (V(C) \cup V(P_2) \cup \ldots \cup V(P_{i-1}))$), and now wish to construct $P_i$. Note that $v_i$ and $v_{2r-i+2}$ are in the same part of the bipartition of $G$; assume without loss of generality that both are in $A$. Let $W_i = \{v_1, \ldots, v_{2r}\} \cup V(P_2) \cup \ldots \cup V(P_{i-1})$. If there exists a vertex $w \in B \setminus W_i$ which is a common neighbour of $v_i$ and $v_{2r-i+2}$, then we set $P_i = (v_i, w, v_{2r-i+2})$; it is evident that $P_i$ satisfies properties (i), (ii), and (iii) in this case. 

Assume then that $X := N_G(v_i) \setminus W_i$ and $Y := N_G(v_{2r-i+2}) \setminus W_i$ are disjoint sets. Since $G$ is bipartite, it follows that $\textrm{deg}_G(v_i, V(C)) \leq r \leq 5$ and $\textrm{deg}_G(v_{2r-i+2}, V(C)) \leq r \leq 5$. Fix any $2 \leq j \leq i-1$. Since $G$ is bipartite, it follows by Property (ii) that 
$$
\max \{\textrm{deg}_G(v_i, V(P_j) \setminus V(C)), \textrm{deg}_G(v_{2r-i+2}, V(P_j) \setminus V(C))\} \leq 2
$$ 
if $j \equiv i \mod 2$, and that 
$$
\max \{\textrm{deg}_G(v_i, V(P_j) \setminus V(C)), \textrm{deg}_G(v_{2r-i+2}, V(P_j) \setminus V(C))\} \leq 1
$$ 
if $j \not\equiv i \mod 2$. Since $r \leq 5$, we conclude that $\max \{\textrm{deg}_G(v_i, W_i), \textrm{deg}_G(v_{2r-i+2}, W_i)\} \leq 9$. Since $\delta(G) \geq k/2 + 10$ holds by the premise of the lemma, it follows that $|X| \geq k/2 + 1$ and that $Y \neq \emptyset$. Let $y \in Y$ be an arbitrary vertex and let $Z = N_G(y, A \setminus W_i)$. Note that $X \cup \{y\} \subseteq B$ and $Z \subseteq A$. Moreover, $\textrm{deg}_G(y, W_i) \leq 8$ holds by Property (iii) and since $r \leq 5$. Since $\delta(G) \geq k/2 + 10$ holds by the premise of the lemma, it follows that $|Z| \geq k/2 + 1$. Hence, by the definition of $\alpha_{\textrm{BIP}}(G)$, there exist $x \in X$ and $z \in Z$ such that $xz \in E(G)$. Set $P_i = (v_i, x, z, y, v_{2r-i+2})$ and note that $P_i$ satisfies properties (i) and (ii). Since no vertex of $V(G) \setminus W_i$ is adjacent to both $v_i$ and $v_{2r-i+2}$, the path $P_i$ satisfies Property (iii) as well.     
\end{proof}

Having proved Theorem~\ref{th::HamBipOAR} and Lemmas~\ref{lem::BipCycleLemma} and~\ref{lem::fewNeighboursBip}, we are in good position to prove Theorem~\ref{th::BipCEcycleSpace}. 

\begin{proof} [Proof of Theorem~\ref{th::BipCEcycleSpace}]
Suppose for a contradiction that $\mathcal{C}_{2n}(G) \neq \mathcal{C}(G)$. We follow the recipe that was presented in Section~\ref{sec::recipe}. That is, we need to handle steps (S1), (S2), and (S3). Our assumption that $\mathcal{C}_{2n}(G) \neq \mathcal{C}(G)$ will then lead to the contradiction appearing in (S5).  

It follows by Theorem~\ref{th::HamBipOAR} that $G$ is Hamiltonian. Let $R$ be as in the premise of Lemma~\ref{lem::subgraphR}; this takes care of (S1). 

Next, we take care of (S2). Starting with (S2a), it follows by Lemma~\ref{lem::BipCycleLemma} that $G$ contains a cycle satisfying properties (a) and (b) from that lemma; let $C = (v_1, \ldots, v_{2r}, v_1)$ be such a cycle. As for (S2b), the required paths $P_2, \ldots, P_r$ are attained via Lemma~\ref{lem::fewNeighboursBip}. 

Finally, we take care of (S3). Let $W = V(P_2) \cup \ldots \cup V(P_r)$; it follows by Property (a) from Lemma~\ref{lem::BipCycleLemma} and by Property (ii) from Lemma~\ref{lem::fewNeighboursBip} that $|W| \leq 20$. Let $G' = G \setminus W$ and note that $G'$ is a balanced bipartite graph. Moreover, it follows by Property (a) from Lemma~\ref{lem::BipCycleLemma} that $v_1$ and $v_{r+1}$ belong to different parts of $G'$. Since $G$ is bipartite, it follows by Lemma~\ref{lem::BipCycleLemma}(a) and by Lemma~\ref{lem::fewNeighboursBip}(ii) that $\max \{\textrm{deg}_G(v_1, W), \textrm{deg}_G(v_{r+1}, W)\} \leq \max \{|A \cap W|, |B \cap W|\} \leq 10$. Moreover, for every vertex $v \in V(G) \setminus W$, it follows by Property (iii) from Lemma~\ref{lem::fewNeighboursBip} that $\textrm{deg}_G(v, W) \leq 8$. Hence, $\delta(G') \geq \delta(G) - 10$ and thus $\alpha_{\textrm{BIP}}(G') \leq \alpha_{\textrm{BIP}}(G) \leq 2 \delta(G) - 24 \leq 2 \delta(G') - 4$. It then follows by Theorem~\ref{th::HamBipOAR} that $G'$ is Hamilton-biconnected; in particular, there is a Hamilton path of $G'$ whose endpoints are $v_1$ and $v_{r+1}$.
\end{proof}

\end{document}